\documentclass[11pt]{article}

\usepackage[a4paper,margin=1.086in]{geometry}
\usepackage{amsmath,amssymb,amsthm,mathtools}
\usepackage{mathrsfs}
\usepackage{enumitem}
\usepackage{hyperref}
\usepackage{microtype}
\usepackage{authblk}
\usepackage{array}
\usepackage{bm}
\usepackage{comment}

\hypersetup{colorlinks=true,linkcolor=blue,citecolor=green,urlcolor=blue}

\newtheorem{theorem}{Theorem}[section]
\newtheorem{proposition}[theorem]{Proposition}
\newtheorem{lemma}[theorem]{Lemma}
\newtheorem{corollary}[theorem]{Corollary}
\theoremstyle{definition}
\newtheorem{definition}[theorem]{Definition}

\theoremstyle{remark}
\newtheorem{remark}[theorem]{Remark}

\usepackage{cleveref}
\crefname{problem}{Problem}{Problems}
\Crefname{problem}{Problem}{Problems}
\crefname{lemma}{Lemma}{Lemmas}
\Crefname{lemma}{Lemma}{Lemmas}
\crefname{proposition}{Proposition}{Propositions}
\Crefname{proposition}{Proposition}{Propositions}
\crefname{theorem}{Theorem}{Theorems}
\Crefname{theorem}{Theorem}{Theorems}
\crefname{corollary}{Corollary}{Corollaries}

\newcommand{\R}{\mathbb{R}}

\newcommand{\ii}{\mathrm{i}}
\newcommand{\diag}{\operatorname{diag}}
\newcommand{\rank}{\operatorname{rank}}

\newcommand{\Ker}{\operatorname{Ker}}
\newcommand{\supp}{\operatorname{supp}}

\newcommand{\norm}[1]{\left\lVert #1\right\rVert}

\newcommand{\wh}{\widehat}
\newcommand{\wt}{\widetilde}
\newcommand{\eps}{\varepsilon}
\newcommand{\cB}{\mathcal B}

\title{Finite Spectral-Band Optimal Control of Acoustic Waves \\ via Subwavelength Point-Like Resonant Actuators}

\author[1]{Arpan Mukherjee}
\author[2]{Mourad Sini }
\affil[1]{MSU-BIT SMBU Joint Research Center of Applied Mathematics, Shenzhen MSU-BIT University, Shenzhen, People's Republic of China}
\affil[2]{Radon Institute for Applied and Computational Mathematics (RICAM), Austrian Academy of Sciences\thanks{The work of M. Sini is partially supported by the Austrian Science Fund (FWF): P 32660 and  P: 36942.}, Linz, Austria}

\date{\today}

\begin{document}
\maketitle

\begin{abstract}
We study the finite-band optimal control of acoustic waves actuated by local clusters of sub-wavelength resonators. The underlying acoustic transmission problem is reduced to a time-domain Foldy--Lax approximation that captures the continuous wave-structure interaction of high-contrast inclusions. Spectral analysis of the associated delayed transfer matrix isolates the principal collective scattering resonances of the clusters. These resonances correspond to weakly damped poles $s_\alpha^\eps=-\gamma_\alpha^\eps+\ii\omega_\alpha^\eps$, with radiation damping $\gamma_\alpha^\eps>0$. Projecting the resulting point-source field onto a finite acoustic spectral band reduces the dynamics to the finite-dimensional coupled system:
\begin{equation*}
\begin{aligned}
        \ddot a+\Lambda a&=C_\eps\eta,\\
        \ddot\eta+2\Gamma_\eps\dot\eta+K_\eps\eta&=u,
\end{aligned}
\end{equation*}
where $a$ denotes the acoustic modal coefficients, $\eta$ represents the collective microstructural states, and $u$ is the control input. This reduced state-space model forms the basis for the subsequent control analysis.
\newline
For a quadratic tracking functional $\mathcal \mathcal{J}_{\mu}$ with regularization parameter $\mu>0$, we prove existence and uniqueness of the optimal control and derive the adjoint optimality system.  The main quantitative contribution is a resonant source-lifting estimate.  If an exact source profile $\eta_r$ is spectrally concentrated in bands $I_\alpha$, then the input
\[
        u_r=(\partial_t^2+2\Gamma_\eps\partial_t+K_\eps)\eta_r
\]
satisfies
\[
        \|u_r\|_{L^2(0,T)}^2
        \le
        \sum_\alpha
        \left(\sup_{\nu\in I_\alpha}
        \left| (\omega_\alpha^\eps)^2+(\gamma_\alpha^\eps)^2-\nu^2
        +2\ii\gamma_\alpha^\eps\nu\right|\right)^2
        \| (\eta_r)_\alpha\|_{\mathcal B_T(I_\alpha)}^2.
\]
This gives the corresponding upper bound for the optimal value.  At exact matching $\nu=\omega_\alpha^\eps$, the multiplier equals $2\gamma_\alpha^\eps\omega_\alpha^\eps+O((\gamma_\alpha^\eps)^3)$; hence local clustering gives a finite resonant gain governed by the real part of the collective pole.  A final section shows that the same small attenuation yields finite-band stabilization, under an explicit modal coupling condition, with decay rate proportional to the cluster-induced damping scale.
\end{abstract}

\noindent


\medskip


\tableofcontents

\section{Introduction}


The optimal control of hyperbolic partial differential equations has been extensively studied because of its theoretical importance and numerous applications in acoustics, vibration control, structural dynamics and wave engineering~\cite{CurtainZwart1995,LasieckaTriggiani2000,Lions1971,Pazy1983}. In many settings, the objective is to determine controls minimizing quadratic performance criteria while steering the system toward a prescribed target or tracking desired trajectories. Depending on the application, the control action may be distributed over a subregion, imposed through the boundary, or concentrated at a finite number of localized devices. Finite-dimensional and finite-band formulations are particularly attractive when only selected modal components are relevant.

\medskip
\noindent
Among localized mechanisms, point or Dirac-type actuators provide a convenient mathematical idealization of highly concentrated sources and have motivated substantial developments in controllability, stabilization and optimization. However, these actuators are typically postulated at the model level rather than derived from an underlying physical mechanism. This observation motivates the present work. We exploit the effective behavior of resonant bubble clusters, whose collective subwavelength dynamics generate localized monopole fields that can be interpreted as internal actuators for the acoustic wave equation. Building on recent asymptotic developments, we investigate the associated finite-band quadratic optimal control problem and quantify the interaction between resonant actuation and optimal control design.

\medskip
\noindent
The control of acoustic waves by small internal devices is usually formulated after an idealization: one prescribes localized sources, boundary inputs, or distributed controls, and then analyzes the corresponding control-to-state map.  For resonant bubbly media there is an additional modelling layer.  The localized sources are not imposed directly; they are generated by the oscillations of small high-contrast bubbles.  Near Minnaert resonance these oscillations are strong, frequency selective, and governed by the multiple scattering between the bubbles.  Hence the actuator is not a static source but a resonant internal system whose parameters are determined by the geometry and material contrast of the bubbles.

\medskip
\noindent
The mathematical modeling of wave propagation in bubbly media has a rich history; see, e.g., the foundational work on effective macroscopic properties \cite{CaflischEtAl1985} and more recent developments in mathematical imaging \cite{DabrowskiGhandricheSini2021}. The underlying acoustic transmission problem can now be rigorously reduced to a time-domain Foldy--Lax system of retarded potentials \cite{MukherjeeSini2023, MukherjeeSini2024, MukherjeeSini2025}. However, while these resonant microstructures act as point-like Dirac actuators \cite{MukherjeeSiniControl2026}, formulating their finite-band optimal control and asymptotic stabilization remains open.

\medskip
\noindent
The starting point of the present work is the time-domain asymptotic description of waves in the presence of finitely many small bubbles.  If $u^\eps$ denotes the acoustic pressure and $u^{\rm in}$ the incident field, then, away from the bubbles and on fixed time intervals, the scattered field is approximated by a retarded Foldy--Lax expansion
\begin{equation}\label{eq:intro-retarded}
        u^\eps(x,t)-u^{\rm in}(x,t)
        =\sum_{i=1}^M G_0(x,z_i)
        q_i^\eps\bigl(t-c_0^{-1}|x-z_i|\bigr)+r^\eps(x,t),
\end{equation}
where the amplitudes $q_i^\eps$ solve a delayed Foldy--Lax system.  The construction used in \cite{MukherjeeSiniControl2026} distributes the bubbles into local clusters: bubbles belonging to the same cluster are close, whereas different clusters are separated at the macroscopic scale.  Consequently, the Laplace-domain Foldy--Lax matrix is dominated by local intra-cluster blocks.  Each block has a principal collective Minnaert pole
\begin{equation}\label{eq:intro-poles}
        s_\alpha^\eps=-\gamma_\alpha^\eps+\ii\omega_\alpha^\eps,
        \qquad \gamma_\alpha^\eps>0,
\end{equation}
which gives one effective collective actuator channel.  The real part $-\gamma_\alpha^\eps$ is small and represents the attenuation created by the clustered scattering mechanism; it is not inserted as an independent damping assumption.

\medskip
\noindent
After projecting the field generated by these collective channels onto a finite acoustic spectral band, one obtains the finite-dimensional wave--actuator system. This finite-band restriction preserves the uniform control properties of the low-frequency dynamics, which are typically lost in unfiltered discrete wave approximations \cite{Zuazua2005}, yielding the coupled system
\begin{equation}\label{eq:intro-main-model}
\begin{aligned}
        \ddot a+\Lambda a&=C_\eps\eta,\\
        \ddot\eta+2\Gamma_\eps\dot\eta+K_\eps\eta&=u .
\end{aligned}
\tag{M}
\end{equation}
Here $a(t)\in\R^n$ are the selected acoustic modal coefficients, $\eta(t)\in\R^{N_c}$ are the principal collective outputs of the local bubble clusters, and $u(t)\in\R^{N_c}$ is the input applied to those collective channels.  The matrices in the second equation are
\begin{equation}\label{eq:intro-GammaK}
        \Gamma_\eps=\diag(\gamma_1^\eps,\ldots,\gamma_{N_c}^\eps),
        \qquad
        K_\eps=\diag\big((\omega_\alpha^\eps)^2+(\gamma_\alpha^\eps)^2\big)_{\alpha=1}^{N_c}.
\end{equation}
The diagonal form in \eqref{eq:intro-GammaK} is obtained only after projection onto the isolated principal collective pole of each local Foldy--Lax block.  The microscopic Foldy--Lax system is not diagonal in the individual bubble amplitudes: it contains the full intra-cluster interaction matrix.  The diagonal matrices in \eqref{eq:intro-GammaK} therefore describe a principal-pole coordinate system, not a microscopic decoupling assumption.  The coupling between the acoustic band and the selected collective channels remains encoded in the generally full matrix $C_\eps$.  This reduction is recalled in detail in Section~\ref{sec:clusters-poles}.

\medskip
\noindent
For comparison, if localized sources at macroscopic cluster centers $y_\alpha$ were available directly, the finite-band model would be
\begin{equation}\label{eq:intro-direct}
        \ddot a+\Lambda a=C_y q,
        \qquad
        (C_y)_{k\alpha}=c_0^2\psi_k(y_\alpha),
\end{equation}
where $q$ is the prescribed point-source strength.  This direct-source viewpoint is the first step in \cite{MukherjeeSiniControl2026}, where the source $q$ is subsequently realized by resonant bubble clusters through a right-inversion argument.  The present paper addresses a complementary question.  The source is the output $\eta$ of the collective resonant channels, while the quantity paid for in the control cost is the input $u$.  Thus the relevant map is no longer $q\mapsto a$ alone, but
\begin{equation}\label{eq:intro-input-output}
        u\longmapsto \eta\longmapsto a.
\end{equation}
The cost reduction caused by resonance is measured through the differential relation
\begin{equation}\label{eq:intro-source-lift}
        u=(\partial_t^2+2\Gamma_\eps\partial_t+K_\eps)\eta.
\end{equation}
For a harmonic component of frequency $\nu$, the multiplier in \eqref{eq:intro-source-lift} is
\begin{equation}\label{eq:intro-main-d}
        d_{\alpha,\eps}(\nu)=
        \left| (\omega_\alpha^\eps)^2+(\gamma_\alpha^\eps)^2-\nu^2
        +2\ii\gamma_\alpha^\eps\nu\right| .
\end{equation}
At frequency matching $\nu=\omega_\alpha^\eps$, this multiplier is of size $2\gamma_\alpha^\eps\omega_\alpha^\eps$ up to higher order terms.  Hence the collective pole produces a finite resonant gain: the cost does not vanish at resonance, but it is small when the cluster-induced damping is small.

\medskip
\noindent
Let $\mathcal R\in\R^{m\times n}$ be an observation matrix and let $a_d\in L^2(0,T;\R^m)$ be a desired finite-band trajectory.  We study the quadratic tracking functional
\begin{equation}\label{eq:intro-functional}
        \mathcal \mathcal{J}_{\mu}(u)=
        \frac12\int_0^T\|\mathcal{R}a(t;u)-a_d(t)\|^2\,dt
        +\frac{\mu}{2}\int_0^T\|u(t)\|^2\,dt,
        \qquad \mu>0,
\end{equation}
subject to \eqref{eq:intro-main-model}.  The parameter $\alpha$ measures the compromise between tracking accuracy and actuator effort.  Resonant tuning improves this compromise because it reduces the norm of the input $u$ required to produce a prescribed effective source output $\eta$. The matrix $\mathcal R \in \mathbb{R}^{m \times n}$ represents the observation (or output) operator for the reduced acoustic system. If
\[
a(t) = (a_1(t),\dots,a_n(t))^\top
\]
denotes the vector of modal coefficients, then the observed quantity is
\[
y(t) =\mathcal R a(t) \in \mathbb{R}^m.
\]
Accordingly, the quadratic cost functional measures the discrepancy between the observed output and a prescribed target trajectory $a_d(t)$ rather than necessarily comparing the full state.

\medskip
\noindent
Several natural choices of the observation matrix arise in practice:
\begin{enumerate}
    \item \textbf{Full-state observation:} Choosing
    \[
   \mathcal R = I_n,
    \]
    where $I_n$ is the identity matrix, yields
    \[
    y(t) = a(t),
    \]
    so that all retained modal amplitudes are tracked simultaneously.

    \item \textbf{Projection onto selected modes:} If only a subset of the modes is of interest, $R$ may be chosen as a coordinate projection that extracts the corresponding components of $a(t)$.

    \item \textbf{Weighted observations:} A diagonal matrix
    \[
    \mathcal R = \operatorname{diag}(w_1,\ldots,w_n)
    \]
    assigns different relative importance to different modes in the tracking objective.

    \item \textbf{Pointwise measurements:} If the acoustic field is observed at spatial locations
    $x_1,\ldots,x_m$, then
    \[
    p(x,t)=\sum_{k=1}^{n} a_k(t)\psi_k(x),
    \]
    and the corresponding observation matrix is
    \[
    \mathcal R_{ik}=\psi_k(x_i),
    \]
    since
    \[
    p(x_i,t)=\sum_{k=1}^{n}\psi_k(x_i)a_k(t).
    \]

    \item \textbf{Regional averages:} More generally, if observations correspond to averages over
    subdomains $\omega_i$, then
    \[
    \mathcal R_{ik}=\int_{\omega_i}\psi_k(x)\,dx,
    \]
    so that each output component represents a spatially averaged measurement.
\end{enumerate}

\noindent
Thus, $\mathcal R$ provides a flexible mechanism for encoding the quantities that are actually observed or required to follow a desired trajectory, ranging from complete modal tracking to localized sensor measurements.

\medskip
\noindent
The main results of this work are: existence and uniqueness of the minimizer of \eqref{eq:intro-functional}; a Pontryagin optimality system for the optimal pair; a source-lifting estimate in terms of the complex-pole multipliers \eqref{eq:intro-main-d}; the corresponding upper bound for the optimal value; and a finite-band stabilization result generated by the small real parts $\gamma_\alpha^\eps$ under an explicit modal coupling condition.  The theorem below records the estimates that will be proved in the sequel.  The notation $\mathcal B_T(I)$ is defined in \Cref{def:BT}.

\medskip
\noindent 
In \cite{MukherjeeSiniControl2026}, the central idea is to show that resonant bubble clusters can approximate prescribed effective Dirac source profiles on finite spectral bands.  Here the effective resonant channels are incorporated into an optimal-control problem: the desired acoustic output is prescribed, the collective source amplitude $\eta$ is an internal variable, and the paid input is the forcing $u$ needed to drive the resonant channels.  The contribution is therefore a quantitative analysis of how the collective pole locations of the bubble clusters enter the quadratic control cost, the detuning estimates, and the finite-band stabilization mechanism.  This distinction is essential for applications, since it identifies not only which source profiles can be produced, but also how efficiently they can be produced through the physical resonant actuator dynamics.

\medskip
\noindent
In this form, the analysis separates four aspects that are often combined in idealized formulations: the finite-band acoustic dynamics, the collective resonant actuator dynamics, the spatial modal coupling matrix, and the frequency-dependent input cost.  This separation makes the role of resonance explicit while keeping the control problem finite-dimensional and fully quantitative.

\begin{theorem}[Main finite-band control and resonant cost estimate]\label{thm:intro-main}
Let $a_d\in L^2(0,T;\R^m)$ and let
\[
        x_0=(a(0),\dot a(0),\eta(0),\dot\eta(0))\in\R^{2n+2N_c}.
\]
\begin{enumerate}[label=(\roman*)]
\item For every $\mu>0$ there is a unique minimizer
\[
        u_{\mu}^{*}\in L^2(0,T;\R^{N_c})
\]
of \eqref{eq:intro-functional} subject to \eqref{eq:intro-main-model}.  It is characterized by the adjoint system \eqref{eq:adjoint} and the stationarity relation
\[
        \alpha u_{\mu}^{*}+p_\zeta=0 .
\]
\item Let $a_r\in H^2(0,T;\R^n)$ be a compatible finite-band reference, i.e. each of its components is in $B_T(I)$, see Definition \ref{def:BT}, and suppose $\rank C_\eps=n$.  Let
\[
        \eta_r=C_\eps^\dagger(\ddot a_r+\Lambda a_r),
        \qquad
        u_r=(\partial_t^2+2\Gamma_\eps\partial_t+K_\eps)\eta_r.
\]
Then
\begin{equation}\label{eq:intro-main-source}
        \|u_r\|_{L^2(0,T)}^2
        \le
        \sum_{\alpha=1}^{N_c}
        \left(\sup_{\nu\in I_\alpha}d_{\alpha,\eps}(\nu)\right)^2
        \|(\eta_r)_\alpha\|_{\mathcal B_T(I_\alpha)}^2,
\end{equation}
where $d_{\alpha,\eps}(\nu)$ is defined in \eqref{eq:intro-main-d}.
\item If $x_0=0$ and $a_d=\mathcal Ra_r$, then the optimal value
\[
        V_\mu(a_d;0):=\inf_u\mathcal \mathcal{J}_{\mu}(u)
\]
satisfies
\begin{equation}\label{eq:intro-main-value}
        V_\mu(a_d;0)
        \le
        \frac{\mu}{2}
        \sum_{\alpha'=1}^{N_c}
        \left(\sup_{\nu\in I_{\alpha'}}d_{\alpha',\eps}(\nu)\right)^2
        \|(\eta_r)_{\alpha'}\|_{\mathcal B_T(I_{\alpha'})}^2.
\end{equation}
\item At exact frequency matching,
\begin{equation}\label{eq:intro-main-match}
        d_{\alpha,\eps}(\omega_\alpha^\eps)
        =2\gamma_\alpha^\eps\omega_\alpha^\eps+O((\gamma_\alpha^\eps)^3).
\end{equation}
Thus, when $\gamma_\alpha^\eps=O(\eps)$, the resonator-mediated input energy needed to generate a matched source is $O(\eps^2)$ relative to the squared source output.
\end{enumerate}
\end{theorem}

\noindent The significance of the above results lies in the explicit connection it establishes between the spectral properties of the resonant actuators and the resulting control cost. Indeed, parts~(ii)--(iv) show that the effort required to generate a prescribed effective source
is governed by the complex-pole multipliers associated with the collective resonances of the
bubble clusters. In particular, near resonance the required control energy is proportional to the
small damping parameter through
\[
d_{\alpha,\varepsilon}(\omega_\alpha^\varepsilon)
    = 2\gamma_\alpha^\varepsilon \omega_\alpha^\varepsilon
      + O\!\left((\gamma_\alpha^\varepsilon)^3\right),
\]
demonstrating quantitatively that resonant tuning reduces the actuator effort needed to produce a desired acoustic response, see Remark \ref{Remark-poles} and Remark \ref{Remark-poles-2} for a more detailed discussion. Thus, taken together, these results show that the present framework extends the standard finite-dimensional quadratic tracking setting by explicitly incorporating the spectral characteristics of the resonant actuators into the control cost. This provides a rigorous bridge between the physics of resonant microstructured actuators and optimal control theory, showing how the location of the collective resonances directly influences tracking performance, control cost, and the efficiency with which desired finite-band wave fields can be generated.
\bigskip

\noindent The remainder of the paper is organized as follows. In Section \ref{sec:model}, we derive the finite-dimensional resonant control model from the effective description of the bubble clusters and introduce the coupled acoustic--actuator dynamics that form the basis of our analysis. Section \ref{sec:LQ} formulates the quadratic tracking problem and establishes the existence, uniqueness, and first-order optimality conditions for the corresponding minimizer. In Section \ref{subsec:detuning}, we analyze the exact source-lifting problem and derive quantitative estimates showing how the resonant pole structure influences the control effort through the associated frequency-dependent multipliers. Section \ref{subsec:cost} applies these estimates to obtain upper bounds for the optimal value function and compares the cost of direct source realization with that of resonator-mediated actuation. Section \ref{sec:stabilization} is devoted to the stabilization properties of the coupled system, where we show that the damping induced by the resonant clusters yields exponential decay under a suitable modal coupling condition. Finally, Section \ref{con} summarizes the main conclusions and discusses perspectives for future work.

\section{Spectral reduction to the finite-dimensional state-space model}\label{sec:model}

This section recalls, in the notation needed below, the reduction from the acoustic transmission problem with bubbles to the finite-dimensional system \eqref{eq:intro-main-model}.  The time-domain Foldy--Lax expansion and the pole analysis are taken from \cite{MukherjeeSini2023,MukherjeeSini2024,MukherjeeSini2025,MukherjeeSiniControl2026}.  In particular, the principal pole structure used below is stated in Proposition 2.1 and proved in Section 5 of \cite{MukherjeeSiniControl2026}; the band-limited right-inversion result is Proposition 2.4 of \cite{MukherjeeSiniControl2026}.

\subsection{The acoustic transmission problem and delayed Foldy–Lax asymptotics}
Let
\[
        D_i^\eps=z_i+\eps B_i,
        \qquad i=1,\ldots,M,
        \qquad
        D^\eps=\bigcup_{i=1}^M D_i^\eps,
\]
be a finite family of small gas bubbles.  Outside $D^\eps$ the parameters are $(\rho_c,\kappa_c)$; inside $D_i^\eps$ they are $(\rho_{b,i}^\eps,\kappa_{b,i}^\eps)$ in the Minnaert scaling used in the time-domain bubble expansions.  The pressure $u^\eps$ solves
\begin{equation}\label{eq:transmission}
        \kappa(x)^{-1}\partial_t^2u^\eps
        -\nabla\cdot\big(\rho(x)^{-1}\nabla u^\eps\big)=F
        \quad\hbox{in }(\mathbb R^3\setminus\partial D^\eps)\times(0,T),
\end{equation}
with continuity of pressure and normal velocity across each $\partial D_i^\eps$, outgoing radiation, and prescribed incident field $u^{\rm in}$ generated by $F$ in the homogeneous background.  Let $c_0=\sqrt{\kappa_c/\rho_c}$ and
\[
        G_0(x,z)=\frac{1}{4\pi |x-z|}.
\]
On compact observation sets separated from the bubbles and for fixed $T$, the scattered field admits the retarded monopole expansion
\begin{equation}\label{eq:time-monopoles}
        u^\eps(x,t)-u^{\rm in}(x,t)
        =\sum_{i=1}^M
        G_0(x,z_i)
        q_i^\eps\bigl(t-c_0^{-1}|x-z_i|\bigr)
        +r^\eps(x,t),
\end{equation}
where $q_i^\eps$ are scalar monopole amplitudes and $r^\eps$ is the asymptotic remainder.

\medskip
\noindent
In the normalization used in \cite{MukherjeeSiniControl2026}, the amplitudes satisfy a delayed Foldy--Lax system of the form
\begin{equation}\label{eq:FL-time-full}
        D_i(\partial_t)q_i^\eps(t)
        +\sum_{j\ne i}q_{ij}^\eps\,
        \partial_t^2 q_j^\eps(t-\tau_{ij})
        =\partial_t^2 f_i(t),
        \qquad i=1,\ldots,M,
\end{equation}
where
\begin{equation}\label{eq:FL-time-terms}
        D_i(\partial_t):=1+\omega_{M,i}^{-2}\partial_t^2,
        \qquad
        \tau_{ij}:=c_0^{-1}|z_i-z_j|,
        \qquad
        q_{ij}^\eps:=\frac{\eps C_j}{4\pi |z_i-z_j|},\quad i\ne j.
\end{equation}
Here $f_i(t)$ denotes the incident trace at $z_i$ after the input normalization.  Thus the time derivative in the interaction term acts only once, namely on the delayed amplitude $q_j^\eps(t-\tau_{ij})$; the coefficient $q_{ij}^\eps$ is a scalar interaction strength and contains no additional time differentiation.  Equivalently,
\begin{equation}\label{eq:FL-time-expanded}
        \left(1+\omega_{M,i}^{-2}\partial_t^2\right)q_i^\eps(t)
        +\sum_{j\ne i}
        \frac{\eps C_j}{4\pi |z_i-z_j|}
        \partial_t^2 q_j^\eps(t-\tau_{ij})
        =\partial_t^2 f_i(t).
\end{equation}
The precise constants depend on the chosen amplitude normalization; the form \eqref{eq:FL-time-expanded} is the one needed here, namely a Minnaert local factor plus delayed Foldy--Lax interactions.

\subsection{Resolvent formulation and the block-dominant transfer matrix}
Let $\wh q^\eps(s)$ and $\wh f(s)$ be Laplace transforms.  Equation \eqref{eq:FL-time-expanded} gives
\begin{equation}\label{eq:FL-Laplace-full}
        \Big(D(s)+s^2Q^\eps(s)\Big)\wh q^\eps(s)=s^2\wh f(s),
\end{equation}
where
\begin{equation}\label{eq:DQ-defs}
        D(s)=\diag\left(1+\frac{s^2}{\omega_{M,i}^2}\right)_{i=1}^M,
        \qquad
        Q_{ij}^\eps(s)=
        \begin{cases}
        \displaystyle \frac{\eps C_j}{4\pi |z_i-z_j|}
        e^{-s|z_i-z_j|/c_0},& i\ne j,\\[2mm]
        0,&i=j.
        \end{cases}
\end{equation}
Equivalently,
\begin{equation}\label{eq:Kb-def}
        \wh q^\eps(s)=H_b^\eps(s)\wh f(s),
        \qquad
        H_b^\eps(s):=\Big(D(s)+s^2Q^\eps(s)\Big)^{-1}s^2.
\end{equation}
With a general transducer-to-bubble trace matrix $G_{\rm tr}(s)$, one writes instead
\begin{equation}\label{eq:full-transfer}
        \wh q^\eps(s)=H_b^\eps(s)G_{\rm tr}(s)\wh\lambda(s).
\end{equation}
The bubbles are grouped into local clusters
\begin{equation}\label{eq:clusters}
        \{1,\ldots,M\}=\bigsqcup_{\alpha=1}^{N_c}I_\alpha.
\end{equation}
The cluster output is
\begin{equation}\label{eq:cluster-output}
        Q_\alpha^\eps(t)=\sum_{i\in I_\alpha}q_i^\eps(t),
        \qquad
        Q^\eps=B_{\rm out}q^\eps,
\end{equation}
where
\begin{equation}\label{eq:Bout}
        (B_{\rm out}q)_\alpha:=\sum_{i\in I_\alpha}q_i.
\end{equation}
The cluster transfer matrix is therefore
\begin{equation}\label{eq:Hcl-def}
        H_{\rm cl}^\eps(s):=B_{\rm out}H_b^\eps(s)G_{\rm tr}(s),
        \qquad
        \wh Q^\eps(s)=H_{\rm cl}^\eps(s)\wh\lambda(s).
\end{equation}

\subsection{Spectral isolation of principal Minnaert poles and collective channels}\label{sec:clusters-poles}
We now explain why the internal part of the reduced model can be written in diagonal form, although the microscopic Foldy--Lax system is fully coupled inside each cluster.  This point is important: the diagonal matrices used later are not imposed at the bubble level; they are obtained after selecting one isolated collective pole from each local Foldy--Lax block.

\medskip
\noindent
For $i,j\in I_\alpha$, $i\ne j$, the construction in \cite{MukherjeeSiniControl2026} uses close intra-cluster distances of order $\eps^p$, $0<p<1$, while $|z_i-z_j|$ is of order one when $i$ and $j$ belong to distinct clusters.  After reordering the bubbles by clusters, the Laplace-domain Foldy--Lax pencil
\begin{equation}\label{eq:FL-pencil}
        A^\eps(s):=D(s)+s^2Q^\eps(s)
\end{equation}
therefore has the block-dominant structure
\begin{equation}\label{eq:block-structure}
        A^\eps(s)=
        \operatorname{diag}\big(A_1^\eps(s),\ldots,A_{N_c}^\eps(s)\big)
        +E_{\rm inter}^\eps(s).
\end{equation}
Here $A_\alpha^\eps(s)$ is the restriction of $A^\eps(s)$ to the bubbles of $I_\alpha$, while $E_{\rm inter}^\eps(s)$ contains the interactions between distinct clusters.  The blocks $A_\alpha^\eps(s)$ are not diagonal: their off-diagonal entries contain the strong local interactions between the bubbles of the same cluster.  The inter-cluster part is lower order in the regime considered in \cite{MukherjeeSiniControl2026} and is treated perturbatively in the pole analysis.

\medskip
\noindent
Proposition 2.1 of \cite{MukherjeeSiniControl2026} gives the following spectral picture for the microscopic transfer matrix.  In the relevant right half-plane, $H_b^\eps(s)$ is meromorphic and its poles are contained in the zero set of
\begin{equation}\label{eq:pole-equation}
        \det A^\eps(s)=0.
\end{equation}
For each local cluster $I_\alpha$ there is a distinguished simple pole
\begin{equation}\label{eq:pole}
        s_\alpha^\eps=-\gamma_\alpha^\eps+\ii\omega_\alpha^\eps,
        \qquad \gamma_\alpha^\eps>0,
\end{equation}
near the corresponding Minnaert frequency.  This pole is separated from the other local poles by the asymptotic cluster gap proved in \cite{MukherjeeSiniControl2026}; moreover, different principal cluster poles are separated at the macroscopic scale when the corresponding Minnaert frequencies are distinct.  The complex conjugate pole $\overline{s_\alpha^\eps}$ is used as well when real-valued time-domain signals are reconstructed.

\medskip
\noindent
Let $r_\alpha^\eps$ and $l_\alpha^\eps$ denote the right and left principal resonant vectors associated with the pole $s_\alpha^\eps$, localized on the cluster $I_\alpha$ and normalized so that the corresponding residue is non-zero.  For a simple pole of the matrix pencil $A^\eps(s)$, the leading singular part of the microscopic transfer matrix has the exact rank-one form
\begin{equation}\label{eq:mic-residue}
        H_b^\eps(s)
        =
        \frac{r_\alpha^\eps\, l_\alpha^\eps{}^*}{s-s_\alpha^\eps}
        \,\kappa_\alpha^\eps
        +H_{\rm reg,\alpha}^\eps(s)
        \qquad\text{near }s_\alpha^\eps,
\end{equation}
where $H_{\rm reg,\alpha}^\varepsilon(s)$ is holomorphic. The scalar factor $\kappa_\alpha^\varepsilon \in \mathbb{C}$ is the generalized derivative trace, defined exactly by the first-order Taylor expansion of the pencil at the pole: 
$$\kappa_\alpha^\varepsilon = \big( l_\alpha^\varepsilon{}^* \frac{\partial A^\varepsilon}{\partial s} \big|_{s_\alpha^\varepsilon} r_\alpha^\varepsilon \big)^{-1}.$$
\noindent
Equivalently, after applying the cluster-output map $B_{\rm out}$ and the transducer trace map $G_{\rm tr}$, the cluster transfer matrix admits the principal expansion
\begin{equation}\label{eq:pole-expansion}
        H_{\rm cl}^\eps(s)=
        \sum_{\alpha=1}^{N_c}
        \frac{b_\alpha^\eps\ell_\alpha^\eps{}^*}{s-s_\alpha^\eps}
        +H_{\rm reg}^\eps(s),
\end{equation}
where $H_{\rm reg}^\eps$ is holomorphic in the corresponding region.  The vectors $b_\alpha^\eps$ and $\ell_\alpha^\eps$ are the projected right and left principal cluster modes; they include the effects of $B_{\rm out}$, the transducers, and the local residue.  Thus the selected actuator channel is a collective mode of the whole cluster, not the motion of a single bubble.

\medskip
\noindent
The principal collective coordinate can be represented, at leading order, by testing the microscopic monopole vector on the corresponding resonant mode.  In time-domain notation we write
\begin{equation}\label{eq:eta-def}
        \eta_\alpha(t)
       :=
        \sum_{i\in I_\alpha}\ell_{\alpha i}^\eps q_i^\eps(t).
\end{equation}
\noindent
The exact geometric normalization is absorbed into the definition of the channel input $u_\alpha$ and into the coupling coefficient appearing in $C_\varepsilon$ below. The rigorous reduction to a finite-dimensional second-order relation on a finite acoustic band is established by the following proposition.

\begin{proposition}
    Let $H_{\rm cl, \alpha}^\varepsilon(s)$ be the exact meromorphic transfer function of the isolated cluster channel $I_\alpha$ on an open domain $\Omega \subset \mathbb{C}$. Then, uniformly on any compact acoustic spectral band $K \subset \ii\mathbb{R}$ enclosing the resonance, the transfer relation admits the second-order canonical representation:
    \begin{equation}\label{eq:theorem-transfer}
        H_{\rm cl, \alpha}^\varepsilon(s) = \frac{\beta_\alpha^\varepsilon}{(s+\gamma_\alpha^\varepsilon)^2 + (\omega_\alpha^\varepsilon)^2}\left(1 + \mathcal{O}(\varepsilon^{1-p})\right) \quad \text{as } \varepsilon \to 0,
    \end{equation}
where $\beta_\alpha^\varepsilon \in \mathbb{R} \setminus \{0\}$ is a real-valued input-output coupling coefficient satisfying satisfying $\lim_{\varepsilon \to 0}\beta_\alpha^\varepsilon = -\omega_{M,\alpha}^2 \rho_\alpha^{(0)}$, with $\rho_\alpha^{(0)} := \big(B_{\rm out} r_\alpha^{(0)}\big) \big(l_\alpha^{(0)}{}^* G_{\rm tr}\big)$.
\end{proposition}
\begin{proof}
We begin by recalling the fact that the local cluster transfer function $H_{\rm cl, \alpha}^\varepsilon(s)$ admits the exact local decomposition
\begin{equation*}
    H_{\rm cl, \alpha}^\varepsilon(s) = \frac{P_\alpha^\varepsilon}{s-s_\alpha^\varepsilon} + H_{\rm reg,\alpha}^\eps(s),
\end{equation*}
where the scalar residue is exactly $P_\alpha^\varepsilon = \kappa_\alpha^\varepsilon (B_{\rm out} r_\alpha^\varepsilon)(l_\alpha^\varepsilon{}^* G_{\rm tr})$.
\newline
We now analyze the asymptotic limits as $\varepsilon \to 0$. We start with analyzing $\kappa_\alpha^\varepsilon = \big( l_\alpha^\varepsilon{}^* \frac{\partial A^\varepsilon}{\partial s} \big|_{s_\alpha^\varepsilon} r_\alpha^\varepsilon \big)^{-1}$ at the unperturbed state ($\varepsilon \to 0$). The unperturbed local block is the diagonal matrix $A_{\alpha\alpha}^{(0)}(s) = \operatorname{diag}(1 + s^2/\omega_{M,i}^2)_{i \in I_\alpha}$ and due to the generic assumption $\omega_{M,i} = \omega_{M,\alpha}$, we derive
\begin{equation*}
    \left. \frac{\partial A^{(0)}}{\partial s} \right|_{s = \ii\omega_{M,\alpha}} = \operatorname{diag}\left( \frac{2s}{\omega_{M,\alpha}^2} \right)_{s = \ii\omega_{M,\alpha}} = \frac{2\ii}{\omega_{M,\alpha}} \mathbf{I}_{|I_\alpha|}.
\end{equation*}
Let $r_\alpha^{(0)}$ and $l_\alpha^{(0)}$ be the unperturbed right and left Perron-Frobenius eigenvectors. As the effective capacitance matrix of the cluster is real and symmetric, it follows that $r_\alpha^{(0)}, l_\alpha^{(0)} \in \mathbb{R}$. Due to $B_{\rm out}, G_{\rm tr} \in \mathbb R$, the unperturbed geometric projection $\rho_\alpha^{(0)} := \big(B_{\rm out} r_\alpha^{(0)}\big) \big(l_\alpha^{(0)}{}^* G_{\rm tr}\big) \in \mathbb{R}$. Adopting the standard inner-product normalization $l_\alpha^{(0)}{}^* r_\alpha^{(0)} = 1$, we further derive that
\begin{equation}\label{eq:kappa-zero}
    \big(\kappa_\alpha^{(0)}\big)^{-1} = l_\alpha^{(0)}{}^* \left( \frac{2\ii}{\omega_{M,\alpha}} \mathbf{I}_{|I_\alpha|} \right) r_\alpha^{(0)} = \frac{2\ii}{\omega_{M,\alpha}} \implies \kappa_\alpha^{(0)} = \frac{\omega_{M,\alpha}}{2\ii} = -\ii\frac{\omega_{M,\alpha}}{2}.
\end{equation}
Crucially, $\kappa_\alpha^{(0)}$ is strictly purely imaginary.
\newline
Now, as proven in \cite{MukherjeeSiniControl2026}, the perturbed eigenvectors satisfy $r_\alpha^\varepsilon = r_\alpha^{(0)} + \mathcal{O}(\varepsilon^{1-p})$ and $l_\alpha^\varepsilon = l_\alpha^{(0)} + \mathcal{O}(\varepsilon^{1-p})$ for $p \in (0,1)$. Consequently, the perturbed residue admits the strict asymptotic expansion
\begin{equation*}
    P_\alpha^\varepsilon = \kappa_\alpha^{(0)} \rho_\alpha^{(0)} + \mathcal{O}(\varepsilon^{1-p}) = -\ii\frac{\omega_{M,\alpha}}{2}\rho_\alpha^{(0)} + \mathcal{O}(\varepsilon^{1-p}).
\end{equation*}
Because the physical time-domain system is real-valued, the principal expansion must include the complex conjugate pole $\overline{s_\alpha^\varepsilon} = -\gamma_\alpha^\varepsilon - \ii\omega_\alpha^\varepsilon$. Combining the conjugate pair yields the rational singular part:
\begin{equation}
    \Pi_\alpha^\varepsilon(s) = \frac{P_\alpha^\varepsilon}{s - s_\alpha^\varepsilon} + \frac{\overline{P_\alpha^\varepsilon}}{s - \overline{s_\alpha^\varepsilon}} = \frac{2\operatorname{Re}(P_\alpha^\varepsilon)(s + \gamma_\alpha^\varepsilon) + 2\operatorname{Im}(P_\alpha^\varepsilon)\omega_\alpha^\varepsilon}{(s+\gamma_\alpha^\varepsilon)^2 + (\omega_\alpha^\varepsilon)^2}.
\end{equation}
Consequently, due to the fact that $\operatorname{Re}(P_\alpha^\varepsilon) = \mathcal{O}(\varepsilon^{1-p})$ and $\operatorname{Im}(P_\alpha^\varepsilon) = \mathcal{O}(1)$, defining $\beta_\alpha^\varepsilon := 2\operatorname{Im}(P_\alpha^\varepsilon)\omega_\alpha^\varepsilon$, the numerator becomes $\mathcal{O}(\varepsilon^{1-p})(s + \gamma_\alpha^\varepsilon) + \beta_\alpha^\varepsilon$. On the compact acoustic band $K \subset \ii\mathbb{R}$, the dynamic factor $|s + \gamma_\alpha^\varepsilon|$ is uniformly bounded, reducing the numerator to $\beta_\alpha^\varepsilon(1 + \mathcal{O}(\varepsilon^{1-p}))$. We also establish the limit of the coupling coefficient as $\varepsilon \to 0$
\begin{equation*}
    \lim_{\varepsilon \to 0} \beta_\alpha^\varepsilon = 2 \left( \lim_{\varepsilon \to 0} \operatorname{Im}(P_\alpha^\varepsilon) \right) \left( \lim_{\varepsilon \to 0} \omega_\alpha^\varepsilon \right) = 2 \left( -\frac{\omega_{M,\alpha}}{2} \rho_\alpha^{(0)} \right) \omega_{M,\alpha} = -\omega_{M,\alpha}^2 \rho_\alpha^{(0)}.
\end{equation*}

\medskip
\noindent
Finally, by continuity, the regular remainder $H_{\rm reg, \alpha}^\varepsilon(s)$ is uniformly bounded by an $\mathcal{O}(1)$ constant on the compact set $K$. Because the distance to the principal pole scales as $|s - s_\alpha^\varepsilon| = \mathcal{O}(\varepsilon^{1-p})$ while the distance to the conjugate pole is macroscopic ($|s - \overline{s_\alpha^\varepsilon}| = \mathcal{O}(1)$), the singular denominator scales strictly as $\mathcal{O}(\varepsilon^{1-p})$. Thus, $|\Pi_\alpha^\varepsilon(s)| \sim \mathcal{O}(\varepsilon^{-(1-p)})$. 
\newline
Following that, due to the uniform ratio bound,
\begin{equation*}
    \left| \frac{H_{\rm reg, \alpha}^\varepsilon(s)}{\Pi_\alpha^\varepsilon(s)} \right| \leq \frac{\mathcal{O}(1)}{\mathcal{O}(\varepsilon^{-(1-p)})} = \mathcal{O}(\varepsilon^{1-p}),
\end{equation*}
factoring out the singular principal part $\Pi_\alpha^\varepsilon(s)$ yields 
\begin{equation*}
    H_{\rm cl, \alpha}^\varepsilon(s) = \Pi_\alpha^\varepsilon(s) \big(1 + \mathcal{O}(\varepsilon^{1-p})\big).
\end{equation*}
Substituting the uniform numerator expansion established beforehand concludes the proof of the canonical representation.
\end{proof}

\noindent
Applying the input renormalization
\begin{equation}\label{eq:u-renormalization}
        u_\alpha:=\beta_\alpha^\varepsilon V_\mu,
\end{equation}
the principal band-projected relation maps directly to the canonical scalar channel
\begin{equation}\label{eq:channel}
        \ddot\eta_\alpha
        +2\gamma_\alpha^\eps\dot\eta_\alpha
        +\big((\omega_\alpha^\eps)^2+(\gamma_\alpha^\eps)^2\big)\eta_\alpha
        =u_\alpha.
\end{equation}
The characteristic roots of \eqref{eq:channel} are exactly $-\gamma_\alpha^\eps\pm\ii\omega_\alpha^\eps$.  The damping coefficient is therefore the negative real part of the collective pole and is inherited from the bubble-cluster scattering mechanism.

\medskip
\noindent
Collecting one principal channel from each cluster gives
\begin{equation}\label{eq:eta-vector-channel}
        \ddot\eta+2\Gamma_\eps\dot\eta+K_\eps\eta=u,
        \qquad
        \Gamma_\eps=\diag(\gamma_1^\eps,\ldots,\gamma_{N_c}^\eps),
\end{equation}
with
\begin{equation}\label{eq:K-vector-channel}
        K_\eps=
        \diag\big((\omega_\alpha^\eps)^2+(\gamma_\alpha^\eps)^2\big)_{\alpha=1}^{N_c}.
\end{equation}
The diagonal structure in \eqref{eq:eta-vector-channel}--\eqref{eq:K-vector-channel} is a consequence of the choice of isolated principal-pole coordinates.  The discarded terms are precisely the regular part of the transfer matrix, the higher local modes, and the lower-order inter-cluster corrections; these will be collected in the remainder $\rho_\eps$ after projection onto the acoustic band.  The coupling between the selected acoustic modes and the selected resonant channels is not diagonal in general and is encoded in the matrix $C_\eps$.

\subsection{Finite-band Galerkin projection and the induced control system}\label{sec:galerkin}

Let $\Omega \subset \mathbb{R}^3$ be a bounded $C^{1,1}$ domain such that the actuator region $D^\varepsilon = \bigcup_{i=1}^M D_i^\varepsilon \Subset \Omega$. For a fixed control horizon $T>0$, we assume the finite propagation constraint
\begin{equation}\label{eq:domain-localization}
    \operatorname{dist}(D^\varepsilon, \partial\Omega) > c_0 T.
\end{equation}
Let $-\Delta_D$ denote the Dirichlet Laplacian on $\Omega$, equipped with an $L^2(\Omega)$-orthonormal eigenbasis $\{\psi_k\}_{k=1}^\infty \subset H_0^1(\Omega) \cap H^2(\Omega)$ and discrete spectrum $0 < \lambda_1 \le \lambda_2 \le \cdots \to \infty$. We define the acoustic angular frequencies $\omega_k := c_0 \sqrt{\lambda_k}$.
\newline
The macroscopic scattered pressure $p^\varepsilon := u^\varepsilon - u^{\rm in}$ satisfies the Dirichlet boundary-value problem
\begin{equation}\label{eq:point-source-wave}
    \begin{cases}
        (\partial_{tt} - c_0^2\Delta)p^\varepsilon(x,t) = c_0^2\sum_{\alpha=1}^{N_c} \beta_\alpha^\varepsilon \eta_\alpha(t) \delta_{y_\alpha}(x), & (x,t) \in \Omega \times (0,T), \\
        p^\varepsilon(x,t) = 0, & (x,t) \in \partial\Omega \times (0,T),
    \end{cases}
\end{equation}
where $y_\alpha \in \Omega$ are the cluster centers and $\eta_\alpha(t) \in \mathbb{R}$ is the canonical collective coordinate of the $\alpha$-th cluster corresponding to the input-output coefficient $\beta_\alpha^\varepsilon$ (see Section \ref{sec:clusters-poles}).
\newline
Given a finite target spectral band $J \subset (0, \infty)$, we define the index set $K_J := \{k \in \mathbb{N} : \omega_k \in J\}$ with cardinality $n := |K_J|$, and the associated finite-dimensional subspace $H_J := \operatorname{span}\{\psi_k\}_{k \in K_J}$. Because $H^2(\Omega) \hookrightarrow C^0(\overline{\Omega})$ in three dimensions, the point evaluations $\psi_k(y_\alpha)$ are well-defined. The Galerkin projection $p_k(t) := \langle p^\varepsilon(\cdot, t), \psi_k \rangle_{L^2(\Omega)}$ reduces \eqref{eq:point-source-wave} to the modal system
\begin{equation}\label{eq:modal-oscillators}
    \ddot{p}_k(t) + \omega_k^2 p_k(t) = c_0^2 \sum_{\alpha=1}^{N_c} \beta_\alpha^\varepsilon \eta_\alpha(t) \psi_k(y_\alpha), \quad \forall k \in K_J.
\end{equation}
Collecting the modal coordinates into the acoustic state vector $a(t) := (p_k(t))_{k\in K_J}^\top \in \mathbb{R}^{n}$, the projected dynamics admit the vectorized second-order representation
\begin{equation}\label{eq:acoustic-state}
    \ddot{a}(t) + \Lambda a(t) = C_\varepsilon \eta(t).
\end{equation} 
Coupling this with the internal macroscopic cluster dynamics \eqref{eq:eta-vector-channel} yields the complete finite-dimensional control system
\begin{equation}\label{eq:induced-system}
    \begin{cases}
        \ddot{a}(t) + \Lambda a(t) = C_\varepsilon \eta(t), \\
        \ddot{\eta}(t) + 2\Gamma_\varepsilon \dot{\eta}(t) + K_\varepsilon \eta(t) = u(t).
    \end{cases}
\end{equation}
Here, the structural matrices $\Lambda \in \mathbb{R}^{n \times n}$ and $C_\varepsilon \in \mathbb{R}^{n \times N_c}$ are strictly defined by
\begin{equation}\label{eq:matrices-def}
    \Lambda := \operatorname{diag}(\omega_k^2)_{k \in K_J}, \qquad 
    (C_\varepsilon)_{k\alpha} := c_0^2 \beta_\alpha^\varepsilon \psi_k(y_\alpha),
\end{equation}
and $u \in L^2(0,T; \mathbb{R}^{N_c})$ acts as the external control input.

\medskip
\noindent
We formulate the coupled finite-band system as an abstract Cauchy problem on a finite-dimensional state space $X$. The generation of the associated continuous semigroup and the properties of the control operator follow from standard linear systems theory \cite{CurtainZwart1995, Pazy1983, TucsnakWeiss2009}.

\section{Linear-quadratic optimal tracking for the finite-band system}\label{sec:LQ}

Let
\[
        x=(a,v,\eta,\zeta):=(a,\dot a,\eta,\dot\eta)
        \in X:=\R^n\times\R^n\times\R^{N_c}\times\R^{N_c}.
\]
Then \eqref{eq:induced-system} is
\begin{equation}\label{eq:first-order}
        \dot x=A_\eps x+Bu,
\end{equation}
with
\begin{equation}\label{eq:AB}
        A_\eps=
        \begin{pmatrix}
        0&I&0&0\\
        -\Lambda&0&C_\eps&0\\
        0&0&0&I\\
        0&0&-K_\eps&-2\Gamma_\eps
        \end{pmatrix},
        \qquad
        B=
        \begin{pmatrix}
        0\\0\\0\\I
        \end{pmatrix}.
\end{equation}
Let $\mathcal R\in\R^{m\times n}$ be an observation matrix and define
\begin{equation}\label{eq:LQ-cost}
        \mathcal{J}_{\mu}(u)=
        \frac12\int_0^T\norm{\mathcal Ra(t;u)-a_d(t)}^2\,dt
        +\frac{\mu}{2}\int_0^T\norm{u(t)}^2\,dt.
\end{equation}

\begin{theorem}\label{thm:LQ}
For every $\mu>0$, $x_0\in X$ and $a_d\in L^2(0,T;\R^m)$, there exists a unique minimizer
\[
        u_{\mu}^{*}\in L^2(0,T;\R^{N_c})
\]
of \eqref{eq:LQ-cost} subject to \eqref{eq:first-order} and $x(0)=x_0$.
\end{theorem}

\begin{proof}
Let $\Pi_a$ denote the projection from the full state
\[
x=(a,v,\eta,\zeta)\in X
\]
onto the acoustic modal component $a\in\mathbb R^n$. For a given control
$u\in L^2(0,T;\mathbb R^{N_c})$, the solution of the finite-dimensional linear system is given by
the variation-of-constants formula
\[
x(t;u)=e^{tA_\varepsilon}x_0
+
\int_0^t e^{(t-s)A_\varepsilon}Bu(s)\,ds .
\]
Therefore
\[
a(t;u)
=
\Pi_a e^{tA_\varepsilon}x_0
+
\Pi_a\int_0^t e^{(t-s)A_\varepsilon}Bu(s)\,ds .
\]
Set
\[
a^0(t):=\Pi_a e^{tA_\varepsilon}x_0
\]
and define
\[
(\mathcal K u)(t)
:=
R\Pi_a\int_0^t e^{(t-s)A_\varepsilon}Bu(s)\,ds .
\]
Then
\[
\mathcal Ra(\cdot;u)=\mathcal Ra^0+\mathcal K u .
\]
The map
\[
\mathcal K:L^2(0,T;\mathbb R^{N_c})
\longrightarrow
L^2(0,T;\mathbb R^m)
\]
is bounded and linear, since the system is finite-dimensional and
$t\mapsto e^{tA_\varepsilon}B$ is bounded on $[0,T]$. Hence the control-to-output map
\[
u\longmapsto\mathcal Ra(\cdot;u)
\]
is continuous and affine. It is affine, rather than linear, because of the free evolution
$\mathcal Ra^0$ generated by the initial condition $x_0$. With this notation, the cost functional can be written as
\[
\mathcal{J}_{\mu}(u)
=
\frac12
\left\|
\mathcal K u-(a_d-\mathcal Ra^0)
\right\|^2_{L^2(0,T;\mathbb R^m)}
+
\frac{\mu}{2}
\|u\|^2_{L^2(0,T;\mathbb R^{N_c})}.
\]
This functional is continuous and convex. Moreover, since $\mu>0$, the term
$
\frac{\mu}{2}\|u\|^2_{L^2(0,T;\mathbb R^{N_c})}
$
is strictly convex. Therefore $\mathcal{J}_{\mu}$ is strictly convex, independently of whether
$\mathcal K$ has a nontrivial kernel. The same term also gives coercivity:
$
\mathcal{J}_{\mu}(u)
\ge
\frac{\mu}{2}
\|u\|^2_{L^2(0,T;\mathbb R^{N_c})}.$
Let $(u_j)$ be a minimizing sequence. By coercivity, $(u_j)$ is bounded in
$L^2(0,T;\mathbb R^{N_c})$. Hence, after extracting a subsequence, there exists
$u_{\mu}^{*}\in L^2(0,T;\mathbb R^{N_c})$ such that $
u_j \rightharpoonup u_{\mu}^{*}
\quad\text{weakly in }L^2(0,T;\mathbb R^{N_c}).
$
Since $\mathcal{J}_{\mu}$ is convex and continuous, it is weakly lower semicontinuous. Therefore
\[
\mathcal{J}_{\mu}(u_{\mu}^{*})
\le
\liminf_{j\to\infty}\mathcal{J}_{\mu}(u_j)
=
\inf_{u\in L^2(0,T;\mathbb R^{N_c})}\mathcal{J}_{\mu}(u).
\]
Hence $u_{\mu}^{*}$ is a minimizer. Finally, the strict convexity of $\mathcal{J}_{\mu}$ implies that
this minimizer is unique.
\end{proof}

\begin{proposition}[Pontryagin optimality system]\label{prop:adjoint}
Let $u_{\mu}^{*}$ be the minimizer and let $(a,v,\eta,\zeta)$ be the corresponding state.  There are adjoint variables $(p_a,p_v,p_\eta,p_\zeta)$ satisfying
\[
        p_a(T)=p_v(T)=p_\eta(T)=p_\zeta(T)=0
\]
and
\begin{equation}\label{eq:adjoint}
\begin{aligned}
        -\dot p_a+
        \Lambda p_v&=R^\top(Ra-a_d),\\
        -\dot p_v-p_a&=0,\\
        -\dot p_\eta-C_\eps^\top p_v+K_\eps p_\zeta&=0,\\
        -\dot p_\zeta-p_\eta+2\Gamma_\eps p_\zeta&=0.
\end{aligned}
\end{equation}
The optimal control is
\begin{equation}\label{eq:opt-control}
        u_{\mu}^{*}(t)=-\alpha^{-1}p_\zeta(t).
\end{equation}
\end{proposition}

\begin{proof}
We recall the finite-dimensional Pontryagin minimum principle in the form used here.  For a system
\[
        \dot x=f(t,x,u),\qquad x(0)=x_0,
\]
and a functional
\[
        J(u)=\int_0^T L(t,x(t),u(t))\,dt+\Phi(x(T)),
\]
if $u^\star$ is optimal and $x^\star$ is the corresponding trajectory, then there is an adjoint variable
$p:[0,T]\to X$ such that, with the Hamiltonian
\[
        \mathcal H(t,x,u,p)=L(t,x,u)+p^\top f(t,x,u),
\]
one has
\begin{equation}\label{eq:PMP-general}
        -\dot p(t)=\partial_x\mathcal H(t,x^\star(t),u^\star(t),p(t)),
        \qquad
        p(T)=\nabla \Phi(x^\star(T)),
\end{equation}
and, for a.e. $t\in(0,T)$,
\begin{equation}\label{eq:PMP-stationarity}
        \partial_u\mathcal H(t,x^\star(t),u^\star(t),p(t))=0 .
\end{equation}
In the present problem,
\[
        f(t,x,u)=A_\eps x+Bu,
        \qquad
        L(t,x,u)=\frac12\|\mathcal Ra-a_d(t)\|^2+\frac{\mu}{2}\|u\|^2,
        \qquad
        \Phi\equiv0.
\]
Thus the Hamiltonian is
\begin{equation}\label{eq:Hamiltonian-proof}
        \mathcal H(t,x,u,p)=
        \frac12\|\mathcal Ra-a_d(t)\|^2+\frac{\mu}{2}\|u\|^2
        +p^\top(A_\eps x+Bu).
\end{equation}
Since there is no terminal cost, the transversality condition in \eqref{eq:PMP-general} gives
\[
        p(T)=0.
\]
The $x$-derivative of \eqref{eq:Hamiltonian-proof} is
\[
        \partial_x\mathcal H(t,x,u,p)=
        A_\eps^\top p+
        \begin{pmatrix}R^\top(\mathcal Ra-a_d(t))\\0\\0\\0\end{pmatrix}.
\]
Consequently the adjoint equation is
\begin{equation}\label{eq:adjoint-vector-proof}
        -\dot p=A_\eps^\top p+
        \begin{pmatrix}\mathcal R^\top(\mathcal Ra-a_d)\\0\\0\\0\end{pmatrix},
        \qquad p(T)=0 .
\end{equation}
The derivative with respect to $u$ gives
\[
        \partial_u\mathcal H(t,x,u,p)=\alpha u+B^\top p.
\]
The stationarity condition \eqref{eq:PMP-stationarity} therefore yields
\begin{equation}\label{eq:stationarity-proof}
        \alpha u_{\mu}^{*}+B^\top p=0 .
\end{equation}
It remains only to write \eqref{eq:adjoint-vector-proof} in components.  Let
\[
        p=(p_a,p_v,p_\eta,p_\zeta),
        \qquad
        p_a,p_v\in\R^n,
        \quad
        p_\eta,p_\zeta\in\R^{N_c}.
\]
From \eqref{eq:AB},
\[
        A_\eps^\top=
        \begin{pmatrix}
        0&-\Lambda&0&0\\
        I&0&0&0\\
        0&C_\eps^\top&0&-K_\eps\\
        0&0&I&-2\Gamma_\eps
        \end{pmatrix},
        \qquad
        B^\top p=p_\zeta .
\]
Substituting this expression in \eqref{eq:adjoint-vector-proof} gives exactly
\[
\begin{aligned}
        -\dot p_a+\Lambda p_v&=\mathcal R^\top(\mathcal Ra-a_d),\\
        -\dot p_v-p_a&=0,\\
        -\dot p_\eta-C_\eps^\top p_v+K_\eps p_\zeta&=0,\\
        -\dot p_\zeta-p_\eta+2\Gamma_\eps p_\zeta&=0,
\end{aligned}
\]
with terminal conditions
\[
        p_a(T)=p_v(T)=p_\eta(T)=p_\zeta(T)=0.
\]
Finally, \eqref{eq:stationarity-proof} and $B^\top p=p_\zeta$ give
\[
        u_{\mu}^{*}(t)=-\alpha^{-1}p_\zeta(t),
\]
which proves the claimed optimality system.
\end{proof}

\section{Resonant Source-Lifting and Quantitative Optimal Bounds}\label{sec:source-lifting-and-cost}

\subsection{Source-Lifting Estimates and Frequency Detuning}\label{subsec:detuning}

Let $a_r\in H^2(0,T;\R^n)$ be a prescribed finite-band trajectory.  An exact collective source output $\eta_r$ must satisfy
\begin{equation}\label{eq:ideal-source}
        C_\eps\eta_r(t)=\ddot a_r(t)+\Lambda a_r(t).
\end{equation}
We impose the finite-band source rank condition
\begin{equation}\label{eq:rank-condition}
        \rank C_\eps=n.
\end{equation}
This is the same modal point-source rank condition as in finite-band Dirac control.  For generic cluster centers $y_\alpha$ and $N_c\ge n$, it holds outside a proper analytic variety.  Let $C_\eps^\dagger$ be a fixed right inverse and set
\begin{equation}\label{eq:eta-r}
        \eta_r=C_\eps^\dagger(\ddot a_r+\Lambda a_r).
\end{equation}
The paid control needed to generate $\eta_r$ through the collective channels is
\begin{equation}\label{eq:u-r}
        u_r=(\partial_t^2+2\Gamma_\eps\partial_t+K_\eps)\eta_r.
\end{equation}

\begin{definition}[Finite-time band class]\label{def:BT}
Let $I\subset(0,\infty)$ be compact.  A function $g\in L^2(0,T)$ belongs to $\cB_T(I)$ if there exists $\wt g\in L^2(\R)$ such that $\wt g|_{(0,T)}=g$ and
\[
        \supp\wh{\wt g}\subset I\cup(-I).
\]
The norm is
\[
        \|g\|_{\cB_T(I)}=
        \inf\{\|\wt g\|_{L^2(\R)}:
        \wt g|_{(0,T)}=g,
        \ \supp\wh{\wt g}\subset I\cup(-I)\}.
\]
\end{definition}
\noindent
For each channel define
\begin{equation}\label{eq:damped-detuning}
        d_{\alpha,\eps}(\nu)=
        \left| (\omega_\alpha^\eps)^2+(\gamma_\alpha^\eps)^2-\nu^2
        +2\ii\gamma_\alpha^\eps\nu\right|.
\end{equation}
For compact bands $I_\alpha$ set
\begin{equation}\label{eq:Delta}
        \Delta_\eps(I_1,\ldots,I_{N_c})=
        \max_{1\le\alpha\le N_c}\sup_{\nu\in I_\alpha}d_{\alpha,\eps}(\nu).
\end{equation}

\begin{theorem}\label{thm:source-lifting}
If $(\eta_r)_\alpha\in\cB_T(I_\alpha)$ for all $\alpha$, then
\begin{equation}\label{eq:source-lifting-estimate}
        \|u_r\|_{L^2(0,T)}^2
        \le
        \sum_{\alpha=1}^{N_c}
        \left(\sup_{\nu\in I_\alpha}d_{\alpha,\eps}(\nu)\right)^2
        \|(\eta_r)_\alpha\|_{\cB_T(I_\alpha)}^2.
\end{equation}
In particular,
\begin{equation}\label{eq:Delta-estimate}
        \|u_r\|_{L^2(0,T)}
        \le
        \Delta_\eps(I_1,\ldots,I_{N_c})
        \|\eta_r\|_{\prod_\alpha\cB_T(I_\alpha)}.
\end{equation}
\end{theorem}

\begin{proof}
Fix $\alpha \in \{1, \dots, N_c\}$. Let $\wt\eta_\alpha \in L^2(\R)$ be an arbitrary extension of $(\eta_r)_\alpha$ matching on $(0,T)$ with $\supp\wh{\wt\eta}_\alpha \subset I_\alpha \cup (-I_\alpha)$. Define 
$$\wt u_\alpha = (\partial_t^2+2\gamma_\alpha^\varepsilon\partial_t+(\omega_\alpha^\varepsilon)^2+(\gamma_\alpha^\varepsilon)^2)\wt\eta_\alpha$$ 
on $\mathbb{R}$, yielding $\wt u_\alpha|_{(0,T)} = (u_r)_\alpha$. With the convention $\widehat{\partial_t g}(\nu)=\ii\nu\widehat g(\nu)$, taking the Fourier transform gives 
$$\wh{\wt u}_\alpha(\nu) = \big((\omega_\alpha^\eps)^2+(\gamma_\alpha^\eps)^2-\nu^2 + 2\ii\gamma_\alpha^\eps\nu\big) \wh{\wt\eta}_\alpha(\nu).$$
Because $\operatorname{supp}(\widehat{\widetilde{\eta}}_\alpha) \subseteq I_\alpha \cup (-I_\alpha)$ and $d_{\alpha,\varepsilon}(\nu)$ is an even function ($d_{\alpha,\varepsilon}(\nu) = d_{\alpha,\varepsilon}(-\nu)$), we can bound the integral using the supremum over the positive frequency interval $I_\alpha$, and Plancherel's theorem yields:
$$ \|\widetilde{u}_\alpha\|_{L^2(\mathbb{R})}^2 = \int_{I_\alpha \cup (-I_\alpha)} |d_{\alpha,\varepsilon}(\nu)|^2 |\widehat{\widetilde{\eta}}_\alpha(\nu)|^2 \,d\nu \le \left(\sup_{\nu \in I_\alpha} d_{\alpha,\varepsilon}(\nu)\right)^2 \|\widetilde{\eta}_\alpha\|_{L^2(\mathbb{R})}^2.$$
By the property of restriction, $\|(u_r)_\alpha\|_{L^2(0,T)}^2 \le \|\widetilde{u}_\alpha\|_{L^2(\mathbb{R})}^2$, yielding:
$$ \|(u_r)_\alpha\|_{L^2(0,T)}^2 \le \left(\sup_{\nu \in I_\alpha} d_{\alpha,\varepsilon}(\nu)\right)^2 \|\widetilde{\eta}_\alpha\|_{L^2(\mathbb{R})}^2 $$
Because the left-hand side is invariant with respect to the choice of the extension, taking the infimum over all admissible extensions $\widetilde{\eta}_\alpha \in L^2(\mathbb{R})$ on the right-hand side preserves the inequality. By the definition of the $\mathcal{B}_T(I_\alpha)$ norm, this gives:
$$ \|(u_r)_\alpha\|_{L^2(0,T)}^2 \le \left(\sup_{\nu \in I_\alpha} d_{\alpha,\varepsilon}(\nu)\right)^2 \|(\eta_r)_\alpha\|_{\mathcal{B}_T(I_\alpha)}^2 $$
Summing over all channels $\alpha=1,\dots,N_c$ and applying the definition of $\Delta_\varepsilon(I_1,\dots,I_{N_c})$ establishes the final bound:
$$ \|u_r\|_{L^2(0,T)} \le \Delta_\varepsilon(I_1,\dots,I_{N_c}) \|\eta_r\|_{\prod_\alpha \mathcal{B}_T(I_\alpha)}. $$
\end{proof}

\noindent
At exact matching,
\begin{equation}\label{eq:exact-match}
        d_{\alpha,\eps}(\omega_\alpha^\eps)=
        \gamma_\alpha^\eps
        \sqrt{(\gamma_\alpha^\eps)^2+4(\omega_\alpha^\eps)^2}
        =2\gamma_\alpha^\eps\omega_\alpha^\eps+O((\gamma_\alpha^\eps)^3).
\end{equation}
Thus the pole real part makes the resonant gain finite.  In the cluster regimes where $\gamma_\alpha^\eps=O(\eps)$, the exact-match input energy is $O(\eps^2)$ relative to the squared output amplitude.

\subsection{Upper Bounds for the Optimal Value Function and Asymptotic Gain}\label{subsec:cost}

Let
\begin{equation}\label{eq:Jmin}
        V_\mu(a_d;x_0):=\inf_{u\in L^2(0,T;\R^{N_c})}\mathcal{J}_{\mu}(u).
\end{equation}

\begin{theorem}\label{thm:optimal-upper}
Let $x_0=0$, $a_d=Ra_r$, and let $\eta_r$ be given by \eqref{eq:eta-r}.  Suppose the compatibility conditions for exact tracking hold and $(\eta_r)_\alpha\in\cB_T(I_\alpha)$.  Then
\begin{equation}\label{eq:optimal-upper}
        V_\mu(a_d;0)
        \le
        \frac{\mu}{2}
        \sum_{\alpha'=1}^{N_c}
        \left(\sup_{\nu\in I_{\alpha'}}d_{\alpha',\eps}(\nu)\right)^2
        \|(\eta_r)_{\alpha'}\|_{\cB_T(I_{\alpha'})}^2.
\end{equation}
Consequently,
\begin{equation}\label{eq:optimal-upper-Delta}
        V_\mu(a_d;0)
        \le
        \frac{\mu}{2}
        \Delta_\eps(I_1,\ldots,I_{N_c})^2
        \|\eta_r\|_{\prod_\alpha\cB_T(I_\alpha)}^2.
\end{equation}
\end{theorem}

\begin{proof}
The input $u_r$ defined in \eqref{eq:u-r} generates $\eta_r$ and therefore $a_r$.  Hence the tracking term vanishes for this admissible control.  Thus
\[
        V_\mu(a_d;0)\le \mathcal{J}_{\mu}(u_r)
        =\frac{\mu}{2}\|u_r\|_{L^2(0,T)}^2.
\]
The result follows from \Cref{thm:source-lifting}.
\end{proof}

\begin{corollary}\label{cor:alpha}
Let $\tau>0$.  Under the hypotheses of \Cref{thm:optimal-upper}, the condition
\begin{equation}\label{eq:alpha-threshold}
        \mu
        \le
        \frac{2\tau}{
        \Delta_\eps(I_1,\ldots,I_{N_c})^2
        \|\eta_r\|_{\prod_\alpha\cB_T(I_\alpha)}^2}
\end{equation}
implies $V_\mu(a_d;0)\le\tau$.  Therefore a smaller complex-pole detuning permits a larger regularization parameter for the same tracking tolerance.
\end{corollary}

\begin{remark}
Corollary~\ref{cor:alpha} shows that the admissible size of the Tikhonov
regularization parameter is determined by the resonant detuning factor
$\Delta_\varepsilon$. In particular, near resonance one has
$\Delta_\varepsilon=O(\gamma^\varepsilon)$, so that smaller damping
permits substantially larger values of $\alpha$ while maintaining the
same tracking tolerance. Thus, beyond reducing the required control
input, appropriately designed resonant actuators enlarge the range of
regularization parameters compatible with a prescribed level of
performance.
\end{remark}

\noindent
For the direct model \eqref{eq:intro-direct}, the paid source needed for $a_r$ is
\begin{equation}\label{eq:q-direct}
        q_r=C_\eps^\dagger(\ddot a_r+\Lambda a_r)=\eta_r.
\end{equation}
The direct energy is $\|q_r\|_{L^2(0,T)}^2$, whereas the resonator-mediated energy is $\|u_r\|_{L^2(0,T)}^2$.

\begin{theorem}\label{thm:comparison}
If $(q_r)_\alpha\in\cB_T(I_\alpha)$, then
\begin{equation}\label{eq:comparison}
        \|u_r\|_{L^2(0,T)}^2
        \le
        \Delta_\eps(I_1,\ldots,I_{N_c})^2
        \|q_r\|_{\prod_\alpha\cB_T(I_\alpha)}^2.
\end{equation}
Thus, on matched collective Minnaert bands, the cost of producing a given effective source is controlled by the small real part of the corresponding cluster pole.
\end{theorem}

\begin{proof}
This is \Cref{thm:source-lifting} with $q_r=\eta_r$.
\end{proof}

\begin{remark}\label{Remark-poles}
The estimate in Theorem~\ref{thm:comparison} quantifies the difference between
direct point-source actuation and resonator-mediated actuation. In the direct model,
the prescribed source $q_r$ is injected into the wave equation itself, so that the
required control effort is measured by $\|q_r\|_{L^2(0,T)}^2$. In contrast, in the
resonator-mediated model the control acts only through the collective resonant dynamics,
and the quantity that is paid for is the input
$
u_r=(\partial_t^2+2\Gamma^\varepsilon\partial_t+K^\varepsilon)\eta_r.
$
Theorem~\ref{thm:comparison} therefore shows that the cost of generating the
same effective source can be bounded by
$
\|u_r\|_{L^2(0,T)}^2
\le
\Delta_\varepsilon^2
\|q_r\|_{L^2(0,T)}^2,
$
where the factor $\Delta_\varepsilon$ is determined entirely by the locations of the
collective scattering resonances. In particular, near resonance one has
$\Delta_\varepsilon=O(\gamma^\varepsilon)$, so that when
$\gamma^\varepsilon=O(\varepsilon)$ the required input energy scales like
$O(\varepsilon^2)$ relative to the direct-source energy. Thus, subwavelength resonators
can generate a prescribed effective source while requiring only a small external input,
with the energy reduction being governed by the damping of the collective resonant modes. To have an idea on the gain we can expect, suppose that producing a prescribed effective source directly requires unit energy,
namely $\|q_r\|_{L^2(0,T)}^2=1$. If the target frequency is tuned to a collective
resonance and $\gamma^\varepsilon=10^{-2}$ with $\omega^\varepsilon=O(1)$, then
$
\Delta_\varepsilon^2
\approx
4(\gamma^\varepsilon)^2(\omega^\varepsilon)^2
=
O(10^{-4}),
$
and Theorem~\ref{thm:comparison} predicts that the corresponding resonator-driven
input satisfies
$
\|u_r\|_{L^2(0,T)}^2
=
O(10^{-4}).
$
This illustrates quantitatively that, once tuned to resonance, the actuator needs only a
very small external input to generate the same effective source.
\end{remark}

\section{Exponential stabilization via cluster-induced radiation damping}\label{sec:stabilization}

Unlike classical space-time multiplier methods \cite{Komornik1994}, we establish exponential stabilization via frequency-domain spectral criteria \cite{Huang1985, Pruss1984}. Since the reduced system is finite-dimensional, verifying the absence of purely imaginary eigenvalues, i.e., $\ii\mathbb{R} \subset \rho(\mathcal{A})$, strictly guarantees exponential decay.
\newline
The same pole real part that regularizes the optimal-control multiplier also produces damping in the homogeneous coupled system. This section records the finite-dimensional stabilization result, demonstrating why the small attenuation generated by local clustering is structurally important.
\newline
Consider the system:
\begin{equation}\label{eq:stab-system}
\begin{cases}
    \ddot a + \Lambda a = C_\eps\eta, \\
    \ddot\eta + 2\Gamma_\eps\dot\eta + K_\eps\eta = C_\eps^\top a.
\end{cases}
\end{equation}
By introducing the augmented state vector $X = (a^\top, \eta^\top)^\top$, this can be expressed as:
\begin{equation}\label{eq:motion}
    \ddot X + \mathbb D_\eps \dot X + \mathbb G_\eps X = 0,
\end{equation}
where the stiffness matrix $\mathbb G_\eps$ and damping matrix $\mathbb D_\eps$ (acting exclusively on the $\eta$-components) are defined as:
\begin{equation}\label{eq:matrices}
    \mathbb G_\eps = \begin{pmatrix} \Lambda & -C_\eps \\ -C_\eps^\top & K_\eps \end{pmatrix}, \quad 
    \mathbb D_\eps = \begin{pmatrix} 0 & 0 \\ 0 & 2\Gamma_\eps \end{pmatrix}.
\end{equation}
\noindent
To establish structural stability, $\mathbb G_\eps$ must be symmetric positive definite. Since the acoustic modal matrix satisfies $\Lambda > 0$, the block LDU factorization yields:
\[
\begin{pmatrix} I & 0 \\ C_\eps^\top\Lambda^{-1} & I \end{pmatrix} 
\begin{pmatrix} \Lambda & -C_\eps \\ -C_\eps^\top & K_\eps \end{pmatrix} 
\begin{pmatrix} I & \Lambda^{-1}C_\eps \\ 0 & I \end{pmatrix} 
= 
\begin{pmatrix} \Lambda & 0 \\ 0 & K_\eps - C_\eps^\top\Lambda^{-1}C_\eps \end{pmatrix}.
\]
Then $\mathbb{G}_\eps > 0$ is guaranteed by a positive Schur complement:
\begin{equation}\label{eq:schur}
    K_\eps - C_\eps^\top\Lambda^{-1}C_\eps > 0.
\end{equation}
Let $\lambda_{\min}$ and $\lambda_{\max}$ denote the strictly positive minimum and maximum eigenvalues of $\mathbb{G}_\eps$. Consequently, the total energy of the system, 
\begin{equation}\label{eq:energy}
    E_\eps(t) = \frac{1}{2}\left(\dot X^\top \dot X + X^\top \mathbb G_\eps X\right),
\end{equation}
is bounded by the standard Euclidean energy $E_{\text{euc}} = |\dot{a}|^2 + |\dot{\eta}|^2 + |a|^2 + |\eta|^2$ via the Rayleigh quotient:
\begin{equation}\label{eq:energy_bounds}
    \min(1, \lambda_{\min}) E_{\text{euc}} \le 2 E_\eps(t) \le \max(1, \lambda_{\max}) E_{\text{euc}}.
\end{equation}
This topological equivalence ensures that $E_\eps(t) \to 0$ if and only if all physical state variables asymptotically decay to zero.

\begin{lemma}\label{lem:dissipation}
Every solution of the coupled system satisfies the dissipation identity:
\begin{equation}\label{eq:dissipation}
    \frac{d}{dt}E_\eps(t) = -2\dot\eta(t)^\top\Gamma_\eps\dot\eta(t).
\end{equation}
\end{lemma}

\begin{proof}
Differentiating $E_\eps(t)$ with respect to time and applying the product rule yields:
\[
    \frac{d}{dt}E_\eps(t) = \dot{X}^\top \ddot{X} + \dot{X}^\top \mathbb{G}_\eps X.
\]
Substituting the acceleration $\ddot{X} = -\mathbb{D}_\eps \dot{X} - \mathbb{G}_\eps X$ from \eqref{eq:motion} gives:
\[
    \frac{d}{dt}E_\eps(t) = \dot{X}^\top (-\mathbb{D}_\eps \dot{X} - \mathbb{G}_\eps X) + \dot{X}^\top \mathbb{G}_\eps X.
\]
The conservative stiffness terms cancel identically, leaving:
\[
    \frac{d}{dt}E_\eps(t) = -\dot{X}^\top \mathbb{D}_\eps \dot{X}.
\]
Expanding $\mathbb{D}_\eps$ back into its block components yields:
\[
    \frac{d}{dt}E_\eps(t) = - \begin{pmatrix} \dot{a} \\ \dot{\eta} \end{pmatrix}^\top \begin{pmatrix} 0 & 0 \\ 0 & 2\Gamma_\eps \end{pmatrix} \begin{pmatrix} \dot{a} \\ \dot{\eta} \end{pmatrix} = -2\dot{\eta}^\top\Gamma_\eps\dot{\eta}.
\]
Since $\Gamma_\eps$ is a diagonal matrix of positive damping coefficients, this derivative is non-positive, completing the proof.
\end{proof}
\noindent
Lemma \ref{lem:dissipation} shows that the energy is non-increasing. Because damping acts exclusively on $\dot\eta$, we apply the finite-dimensional LaSalle/Huang--Pr\"uss principle to establish global exponential stability.

\begin{lemma}\label{lem:LH}
Let $\dot z=\mathcal A z$ on $\R^d$.  Suppose that there exists a quadratic energy
\[
        \mathcal E(z)=\frac12 z^\top Pz,\qquad P=P^\top>0,
\]
such that along all solutions
\[
        \frac{d}{dt}\mathcal E(z(t))=-z(t)^\top\mathcal D z(t)\le0,
        \qquad \mathcal D=\mathcal D^\top\ge0.
\]
Then the following statements are equivalent.
\begin{enumerate}[label=(\roman*)]
\item There exist $M\ge1$ and $\delta>0$ such that
\[
        \mathcal E(z(t))\le M e^{-\delta t}\mathcal E(z(0)),
        \qquad t\ge0.
\]
\item The matrix $\mathcal A$ has no eigenvalue on the imaginary axis $\ii\R$.
\item The largest $\mathcal A$-invariant subspace contained in $\Ker\mathcal D$ is $\{0\}$.
\end{enumerate}
\end{lemma}

\begin{proof}
To prove the equivalence of statements (i), (ii), and (iii) under the finite-dimensional LaSalle/Huang--Pr\"uss principle, we map the dissipation trajectory onto the Continuous Algebraic Lyapunov Equation (CALE), see (\ref{CALE}) below.
\newline
Consider the abstract energy functional $\mathcal E(z) = \frac{1}{2}z^\top Pz$ along the linear trajectory $\dot{z} = \mathcal Az$:
$$ \frac{d}{dt}\mathcal E(z(t)) = \frac{1}{2}\dot{z}^\top Pz + \frac{1}{2}z^\top P\dot{z} = \frac{1}{2}z^\top (\mathcal{A}^\top P + P\mathcal{A})z.$$
\newline
Equating this to the defined dissipation $-z^\top \mathcal Dz$ for all vectors $z \in \mathbb{R}^d$ yields the identity:
\begin{equation}\label{CALE} \mathcal{A}^\top P + P\mathcal{A} = -2\mathcal{D}.\end{equation}
Let $(\lambda, v)$ be any eigenpair of $\mathcal{A}$, such that $\mathcal{A}v = \lambda v$ for $v \in \mathbb{C}^d \setminus \{0\}$. Multiplying identity (\ref{CALE}) on the left by the conjugate transpose $v^*$ and on the right by $v$ gives:
$$ v^*(\mathcal{A}^\top P + P\mathcal{A})v = -2v^*\mathcal{D}v \implies \bar{\lambda}v^*Pv + \lambda v^*Pv = -2v^*\mathcal{D}v,$$
which further implies
$$ 2\Re(\lambda)(v^*Pv) = -2v^*\mathcal{D}v \implies \Re(\lambda) = -\frac{v^*\mathcal{D}v}{v^*Pv} $$
Because $P > 0$ and $\mathcal{D} \ge 0$, we have $v^*Pv > 0$ and $v^*\mathcal{D}v \ge 0$. Consequently, $\Re(\lambda) \le 0$ for all $\lambda \in \sigma(\mathcal{A})$.
\paragraph{Equivalence of (i) and (ii):}
Because $P$ is symmetric positive definite $\big(P>0\big)$, its eigenvalues satisfy $0 < \lambda_{\min}(P) \le \lambda_{\max}(P) < \infty$. The energy functional is strictly bounded by the standard Euclidean norm $\|z\|^2$:
$$ \frac{1}{2}\lambda_{\min}(P)\|z\|^2 \le \mathcal E(z) \le \frac{1}{2}\lambda_{\max}(P)\|z\|^2 $$
Thus, the exponential decay of $\mathcal{E}(z(t))$ is topologically equivalent to the exponential stability of the state trajectory $z(t) = e^{\mathcal{A}t}z(0)$. In finite dimensions, $\dot{z} = \mathcal{A}z$ is exponentially stable if and only if $\Re\lambda < 0$ for all eigenvalues. Given that $\Re\lambda \le 0$ globally, $\Re\lambda < 0$ holds if and only if $\sigma(\mathcal{A}) \cap \mathrm{i}\mathbb{R} = \emptyset$.
\paragraph{Proof that (ii) implies (iii):}
We proceed via a contradiction argument. Assume (iii) is false; therefore, the maximal $\mathcal{A}$-invariant subspace contained in $\ker D$, denoted $\mathcal{M}$, is non-trivial ($\mathcal{M} \neq \{0\}$).
\newline
As an invariant subspace, the restriction $\mathcal{A}|_{\mathcal{M}}$ possesses at least one eigenvalue $\lambda_0 \in \sigma(\mathcal{A})$ with a corresponding eigenvector $v_0 \in \mathcal{M} \setminus \{0\}$. Because $\mathcal{M} \subset \ker \mathcal{D}$, $\mathcal{D}v_0 = 0$, leading to $v_0^*Dv_0 = 0$. Substituting this into the spectral real-part equation forces:
$$ \Re(\lambda_0) = -\frac{0}{v_0^*Pv_0} = 0 $$
Thus, $\lambda_0$ is purely imaginary ($\lambda_0 \in \mathrm{i}\mathbb{R}$), contradicting statement (ii).
\paragraph{Proof that (iii) implies (ii):}
Assuming (ii) is false, there exists a purely imaginary eigenvalue $\lambda = \mathrm{i}\mu$ ($\mu \in \mathbb{R}$) with an eigenvector $v \in \mathbb{C}^d \setminus \{0\}$. The spectral identity dictates:
$$ -\frac{v^*\mathcal{D}v}{v^*Pv} = 0 \implies v^*\mathcal{D}v = 0 $$
Since $\mathcal{D} \ge 0$, this implies $v \in \ker \mathcal{D}$. Decomposing $v = v_R + \mathrm{i}v_I$ for $v_R, v_I \in \mathbb{R}^d$, we find $\mathcal{D}v_R = 0$ and $\mathcal{D}v_I = 0$, situating both in $\ker D$. Evaluating $A$ on these components gives:
$$ A(v_R + \mathrm{i}v_I) = \mathrm{i}\mu(v_R + \mathrm{i}v_I) = -\mu v_I + \mathrm{i}\mu v_R $$
Equating real and imaginary parts reveals $\mathcal{A}v_R = -\mu v_I$ and $\mathcal{A}v_I = \mu v_R$. Thus, the non-trivial subspace $\mathcal{W} = \text{span}\{v_R, v_I\}$ is strictly $\mathcal{A}$-invariant and fully contained within $\ker \mathcal{D}$. This contradicts statement (iii).
\end{proof}

\noindent The damping acts only through $\dot\eta$.  The next theorem verifies the invariant-set condition of Lemma \ref{lem:LH}.  The condition says that no acoustic eigenmode in the selected finite band is invisible to the collective cluster channels.

\begin{theorem}\label{thm:stab}
Suppose \eqref{eq:schur} holds, $\Gamma_\eps>0$, and
\begin{equation}\label{eq:modal-observability}
        \Ker(\Lambda-\lambda^2 I)\cap\Ker(C_\eps^\top)=\{0\}
        \quad\text{for every }\lambda^2\in\sigma(\Lambda).
\end{equation}
Then there exist $M_\eps\ge1$ and $\delta_\eps>0$ such that every solution of \eqref{eq:stab-system} satisfies
\begin{equation}\label{eq:exp-stab}
        E_\eps(t)\le M_\eps e^{-\delta_\eps t}E_\eps(0),
        \qquad t\ge0.
\end{equation}
Moreover, for fixed $\Lambda,K_\eps,C_\eps$ and $\Gamma_\eps=\gamma_\eps\Gamma_0$ with $\Gamma_0>0$ and $0<\gamma_\eps\ll1$, one has
\begin{equation}\label{eq:rate-order}
        \delta_\eps=O(\gamma_\eps).
\end{equation}
\end{theorem}

\noindent The condition (\ref{eq:modal-observability}) is a finite-dimensional observability (or non-invisibility) assumption for the reduced acoustic modes. It requires that no nontrivial eigenmode of the acoustic operator contained in the selected spectral band be orthogonal to all resonant actuator channels. Equivalently, every eigenspace of $\Lambda$ must be detected by the coupling matrix $C_\varepsilon$. This means that no retained acoustic eigenfunction (or linear combination of eigenfunctions corresponding to the same eigenvalue) vanishes simultaneously at all cluster centers $y_\alpha$. Such a condition is generic with respect to the placement of the resonator clusters and excludes only exceptional configurations in which an acoustic mode is completely decoupled from the actuators.

\begin{proof}
Set
\[
        q=\begin{pmatrix}a\\ \eta\end{pmatrix},
        \qquad
        \mathbb D_\eps=\begin{pmatrix}0&0\\0&2\Gamma_\eps\end{pmatrix},
        \qquad
        \mathbb G_\eps=\begin{pmatrix}\Lambda&-C_\eps\\-C_\eps^\top&K_\eps\end{pmatrix}.
\]
Then \eqref{eq:stab-system} can be written as
\begin{equation}\label{eq:second-order-stab-proof}
        \ddot q+ \mathbb D_\eps\dot q+\mathbb G_\eps q=0.
\end{equation}
By the Schur complement condition \eqref{eq:schur}, $\mathbb G_\eps>0$; hence the energy \eqref{eq:energy} is a positive definite quadratic form in $(q,\dot q)$.  By Lemma \ref{lem:dissipation},
\[
        \frac{d}{dt}E_\eps(t)=-2\dot\eta^\top\Gamma_\eps\dot\eta\le0.
\]
Thus the hypotheses of Lemma \ref{lem:LH} are satisfied for the first-order form of \eqref{eq:second-order-stab-proof}.  It remains to prove that the largest invariant set contained in the zero-dissipation set is trivial.
\newline
Let a trajectory satisfy
\[
        \dot\eta(t)^\top\Gamma_\eps\dot\eta(t)=0,
        \qquad t\ge0.
\]  
Since \(\Gamma_\varepsilon>0\), the identity
\[
\dot\eta(t)^\top\Gamma_\varepsilon\dot\eta(t)=0
\]
implies \(\dot\eta(t)=0\), hence \(\eta(t)=\eta_0\). Substituting this into
\[
\ddot q+ \mathbb D_\varepsilon\dot q+\mathbb G_\varepsilon q=0
\]
gives
\[
\ddot a+\Lambda a-C^\varepsilon\eta_0=0,
\qquad
K^\varepsilon\eta_0-(C^\varepsilon)^\top a=0.
\]
The second identity implies \((C^\varepsilon)^\top a(t)\) is constant. Therefore the only possible zero-dissipation trajectories are undamped acoustic oscillations satisfying
\[
\ddot a+\Lambda a=0,\qquad (C^\varepsilon)^\top a(t)=0.
\]
By the observability condition
\[
\ker(\Lambda-\lambda^2 I)\cap\ker((C^\varepsilon)^\top)=\{0\},
\]
such a trajectory must be \(a\equiv0\). Then
\[
K^\varepsilon\eta_0=(C^\varepsilon)^\top a=0.
\]
Since \(K^\varepsilon>0\), it follows that \(\eta_0=0\). Hence the largest invariant set contained in the zero-dissipation set is trivial. Consequently 
\begin{equation}\label{eq:undamped-a-proof}
        \ddot a+\Lambda a=0,
        \qquad
        C_\eps^\top a(t)=0,
        \qquad \eta(t)=0.
\end{equation}
Use the spectral decomposition of $\Lambda$:
\[
        \R^n=\bigoplus_{\lambda^2\in\sigma(\Lambda)}E_\lambda,
        \qquad E_\lambda=\Ker(\Lambda-\lambda^2I).
\]
Every solution of $\ddot a+\Lambda a=0$ has the form
\[
        a(t)=\sum_{\lambda^2\in\sigma(\Lambda)}
        \left(\cos(\lambda t)a_\lambda+\frac{\sin(\lambda t)}{\mu}b_\lambda\right),
        \qquad a_\lambda,b_\lambda\in E_\lambda .
\]
The identity $C_\eps^\top a(t)=0$ for every $t$ implies, by linear independence of the functions $\cos(\lambda t)$ and $\sin(\lambda t)$ for the distinct positive frequencies $\lambda$, that
\[
        C_\eps^\top a_\lambda=0,
        \qquad
        C_\eps^\top b_\lambda=0
        \qquad\hbox{for every }\lambda^2\in\sigma(\Lambda).
\]
The modal coupling condition \eqref{eq:modal-observability} gives $a_\lambda=b_\lambda=0$ for every $\lambda$.  Hence $a\equiv0$ and, by \eqref{eq:undamped-a-proof}, $\eta\equiv0$.  The invariant subset of the zero-dissipation set is therefore $\{0\}$.
\newline
By Lemma \ref{lem:LH}, the first-order semigroup generated by \eqref{eq:stab-system} is exponentially stable.  Since $E_\eps$ is equivalent to the Euclidean energy, there exist $M_\eps\ge1$ and $\delta_\eps>0$ such that \eqref{eq:exp-stab} holds.
\newline
It remains to discuss the dependence on the small cluster-induced damping.  Let $\Gamma_\eps=\gamma_\eps\Gamma_0$ with $\Gamma_0>0$.  The first-order matrix associated with \eqref{eq:stab-system} depends analytically on $\gamma_\eps$.  At $\gamma_\eps=0$ the system is conservative and its spectrum lies on the imaginary axis.  The perturbation in the generator is linear in $\gamma_\eps$ and acts only on the collective velocities.  Therefore the spectral abscissa can move into the left half-plane only at the scale $O(\gamma_\eps)$.  For every fixed $\gamma_\eps>0$, the invariant-set argument above excludes eigenvalues on $\ii\R$; hence the spectral abscissa is strictly negative.  This proves that the exponential decay rate has the natural order
\[
        \delta_\eps=O(\gamma_\eps),
\]
which is \eqref{eq:rate-order}.
\end{proof}

\begin{remark}\label{Remark-poles-2}
The decay in \Cref{thm:stab} is generated by the real parts $\gamma_\alpha^\eps$ of the collective poles.  Resonant tuning affects stabilization in two ways: it increases the transfer of acoustic energy to the damped collective channels through the same frequency matching that appears in the multiplier $d_{\alpha,\eps}$, and it determines the size of the damping scale.  Since $\gamma_\alpha^\eps$ is small for small bubbles, the finite-band decay is weak, typically of order $O(\eps)$ in the cluster regimes considered in \cite{MukherjeeSiniControl2026}.
\end{remark}

\section{Conclusion}\label{con}

Starting from the bubbly acoustic transmission problem, the time-domain Foldy--Lax approximation gives
\[
        u^\eps-u^{\rm in}
        =\sum_iG_0(\cdot,z_i)q_i^\eps(\cdot-c_0^{-1}|\cdot-z_i|)+r^\eps,
\]
where $q^\eps$ solves the delayed system \eqref{eq:FL-time-expanded}.  The local-cluster geometry turns the Laplace-domain Foldy--Lax matrix into a block-dominant matrix, and Proposition 2.1 of \cite{MukherjeeSiniControl2026} provides principal collective poles
\[
        s_\alpha^\eps=-\gamma_\alpha^\eps+\ii\omega_\alpha^\eps .
\]
Projection onto the associated collective channels and onto a finite acoustic band yields the system
\[
        \ddot a+\Lambda a=C_\eps\eta,
        \qquad
        \ddot\eta+2\Gamma_\eps\dot\eta+K_\eps\eta=u.
\]
For the quadratic functional \eqref{eq:LQ-cost}, the minimizer exists uniquely and satisfies the adjoint system \eqref{eq:adjoint}--\eqref{eq:opt-control}.  The finite-band source-lifting estimate is
\[
        \|u_r\|_{L^2(0,T)}^2
        \le
        \sum_\alpha
        \left(\sup_{\nu\in I_\alpha}d_{\alpha,\eps}(\nu)\right)^2
        \| (\eta_r)_\alpha\|_{\mathcal B_T(I_\alpha)}^2,
\]
with
\[
        d_{\alpha,\eps}(\nu)=
        \left| (\omega_\alpha^\eps)^2+(\gamma_\alpha^\eps)^2-\nu^2
        +2\ii\gamma_\alpha^\eps\nu\right|.
\]
This gives the optimal-value bound \eqref{eq:optimal-upper}.  At frequency matching, $d_{\alpha,\eps}(\omega_\alpha^\eps)=2\gamma_\alpha^\eps\omega_\alpha^\eps+O((\gamma_\alpha^\eps)^3)$, which quantifies the finite resonant gain.  The same pole real parts also yield the finite-band stabilization identity \eqref{eq:dissipation} and, under the modal coupling condition \eqref{eq:modal-observability}, exponential decay with rate proportional to the cluster-induced damping scale.
\newline
The analysis therefore shows that the collective Minnaert poles play a double role.  Their imaginary parts determine the frequency windows in which the actuator channels couple efficiently to the selected acoustic modes, while their real parts determine both the finite gain in the source-lifting estimate and the stabilization scale.  In this sense the resonant microstructure is not treated as a passive scattering correction to an otherwise ideal actuator, but as part of the control mechanism itself.  The framework suggests a systematic route for designing finite-band wave controls in which the geometry and material parameters of the clusters are chosen together with the optimal input, rather than after the control problem has been idealized.
\medskip

\noindent An interesting feature of the present framework is that the same collective scattering poles govern both the efficiency of resonator-mediated actuation and the intrinsic stabilization of the reduced dynamics through their nonzero real parts.

\bigskip

\noindent
\textit{\textbf{Acknowledgments.}}
\\
\textit{\textbf{Data Availability Statement.}} Data sharing is not applicable to this article as no datasets were generated or analyzed during the current study.
\bigskip

\noindent
\textit{\textbf{Declarations.}}
\newline
\textit{\textbf{Conflict of interest.}} The authors declare that they have no conflict of interest.

\end{document}